\documentclass[a4paper,11pt]{amsart}
\usepackage[utf8]{inputenc}
\usepackage{amsmath,amssymb,amsfonts,amsthm,bm}
\usepackage{hyperref}
\usepackage{graphicx}
\usepackage{enumerate}
\usepackage{geometry}
\usepackage{color}
\usepackage{parskip}
\usepackage{enumitem} % Customize lists
\newcommand*{\abs}[1]{\lvert#1\rvert}
\newcommand*{\concat}{\symbol{94}}

\newcommand*{\nbd}{\nobreakdash-\hspace{0pt}}

\DeclareMathOperator{\cyc}{Cycle}
\DeclareMathOperator{\dc}{DiagComp}
\DeclareMathOperator{\cycCount}{CycCount}   
\newcommand{\cT}{\mathcal{T}}
\newcommand{\res}{\!\upharpoonright\!}
\newcommand{\tup}[1]{\langle #1 \rangle}

\newtheorem{theorem}{Theorem}[section]
\newtheorem{lemma}[theorem]{Lemma}

\newtheorem{corollary}[theorem]{Corollary}
\newtheorem{question}[theorem]{Question}

\theoremstyle{definition}
\newtheorem{definition}[theorem]{Definition}
\newtheorem{remark}[theorem]{Remark}

\begin{document}
\title{Bi-Isolated d.c.e.\ Degrees and \(\Sigma_1\) Induction}

\author[Yong Liu]{Yong Liu}
\address[Yong Liu]{School of Information Engineering\\
Nanjing Xiaozhuang University\\
CHINA}
\email{\href{mailto:liuyong@njxzc.edu.cn}{liuyong@njxzc.edu.cn}}

\author[Peng]{Cheng Peng}
\address[Peng]{Institute of Mathematics\\
Hebei University of Technology\\
CHINA}
\email{\href{mailto:pengcheng@hebut.edu.cn}{pengcheng@hebut.edu.cn}}
	


\thanks{Peng's research was partially funded by the Science and
Technology Project of the Hebei Education Department (No.~QN2023009) and the National Natural Science Foundation of China (No.~12271264). Yong Liu's research was partially funded by Nanjing Xiaozhuang University (No.~2022NXY39).}
	
\begin{abstract}
A Turing degree is d.c.e.\ if it contains a set that is the difference of two c.e.\ sets. A d.c.e.\ degree \(\bm{d}\) is \emph{isolated} if there exists a c.e.\ degree \(\bm{a}<\bm{d}\) such that every c.e.\ degree below \(\bm{d}\) is also below \(\bm{a}\); \(\bm{d}\) is \emph{upper isolated} if there exists a c.e.\ degree \(\bm{a}>\bm{d}\) such that every c.e.\ degree above \(\bm{d}\) is also above \(\bm{a}\); \(\bm{d}\) is \emph{bi-isolated} if it is both isolated and upper isolated. 
In this paper, we prove the existence of bi-isolated d.c.e.\ degrees in models of \(\mathsf{I}\Sigma_1\).
\end{abstract}

\maketitle

\section{Introduction}
The Ershov hierarchy~\cite{RN134,RN135,RN133} for \(\Delta_2^0\) degrees has been extensively studied in the literature. A set \(A\) is \(n\)-c.e.\ (\(n\)-\emph{computably enumerable}) if there exists a computable function \(f\) such that, for all \(x\), \(\lim_{s} f(x,s) = A(x)\), \(f(x,0) = 0\), and \(|\{s \mid f(x,s+1) \neq f(x,s)\}| \le n\). 
In this way, the \(1\)-c.e.\ sets coincide with the usual definition of the c.e.\ sets; \(2\)-c.e.\ sets coincide with the d.c.e.\ sets, which is the difference of two c.e.\ sets.
A degree is an \(n\)-c.e.\ degree if it contains an \(n\)-c.e.\ set, and we are mostly interested in comparing the structure of \(n\)-c.e.\ degrees with that of \(m\)-c.e.\ degrees to see whether they differ. 
Along this line, earlier results show that a maximal degree (that is, an incomplete degree with only \(\bm{0}'\) above it) exists among the d.c.e.\ degrees~\cite{RN119,RN139}, but not among the c.e.\ degrees~\cite{RN153}.
In other words, while the c.e.\ degrees are dense, the d.c.e.\ degrees are not. However, various kinds of weak density results do hold (for example, \cite{RN359,RN362,RN87}). 
In particular, Cooper and Yi~\cite{RN362} proved that for any c.e.\ degree \(\bm{a}\) and d.c.e.\ degree \(\bm{d}\) with \(\bm{a}<\bm{d}\), there exists some d.c.e.\ degree \(\bm{c}\) strictly between \(\bm{a}\) and \(\bm{d}\). 
They also proved that it is necessary for \(\bm{c}\) to be d.c.e.\ by introducing the notion of isolated d.c.e.\ degrees and proving that such a degree exists.
\begin{definition}\label{def:isolated_below}
    A d.c.e.\ degree \(\bm{d}\) is \emph{isolated} if there exists a c.e.\ degree \(\bm{a}<\bm{d}\) such that every c.e.\ degree below \(\bm{d}\) is also below \(\bm{a}\).
\end{definition}

Efremov~\cite{RN360, RN361} and Wu~\cite{RN172} then studied similar notions and showed, respectively, the existence of an upper isolated d.c.e.\ degree and the existence of a bi-isolated d.c.e.\ degree.

\begin{definition}\label{def:isolated_above}
    A d.c.e.\ degree \(\bm{d}\) is \emph{upper isolated} if there exists a c.e.\ degree \(\bm{a}>\bm{d}\) such that every c.e.\ degree above \(\bm{d}\) is also above \(\bm{a}\). A d.c.e.\ degree \(\bm{d}\) is \emph{bi-isolated} if it is both isolated and upper isolated.
\end{definition}

Their results were further strengthened by Liu~\cite{RN70}, who showed that there exists an isolated maximal d.c.e.\ degree, which is clearly bi-isolated.

From the perspective of reverse mathematics, one seeks to clarify the proof-theoretic strength required to establish a given theorem. In the context of reverse recursion theory, Paris and Kirby~\cite{RN319} introduced a hierarchy of fragments of Peano arithmetic, which particularly considers different levels of induction schemes and bounding schemes. For our purposes, we consider schemes not beyond \(\mathsf{I}\Sigma_1\) and provide only a brief overview; for a comprehensive treatment, see Chong, Li, and Yang~\cite{RN100}.

We work in the language of first-order arithmetic \(\mathcal{L}=\{0,1,+,\times, \exp, <\}\), where \(\exp\) denotes the exponential function \(x\mapsto 2^x\). 
To study recursion theory, one often works in the base theory \(\mathsf{P}^- + \mathsf{I}\Sigma_0+\mathsf{B}\Sigma_1+\mathsf{Exp}\), where \(\mathsf{P}^-\) denotes \(\mathsf{PA}\) without induction schemes, \(\mathsf{I}\Sigma_0\) denotes the induction scheme for \(\Sigma_0\) formulas, \(\mathsf{B}\Sigma_1\) denotes the bounding scheme for \(\Sigma_1\) formulas, and \(\mathsf{Exp}\) denotes the axiom \(\forall x \exists y (y = \exp(x))\). 
Over this base theory, Li~\cite{RN101} showed that the existence of a proper d.c.e.\ degree below \(\bm{0}'\) is equivalent to \(\mathsf{I}\Sigma_1\). 
To construct a d.c.e.\ degree with additional properties, one may have to use infinite injury priority methods, which apparently demand more inductive strength. 
However, it was proved~\cite{RN356} that isolated and upper isolated d.c.e.\ degrees exist in models of \(\mathsf{I}\Sigma_1\), leaving open the question of whether bi-isolated d.c.e.\ degrees exist in models of \(\mathsf{I}\Sigma_1\). We answer this question affirmatively in this paper.

\begin{theorem}\label{thm:main1}
    Bi-isolated d.c.e.\ degrees exist in every model of \(\mathsf{I}\Sigma_1\).
\end{theorem}
In light of Li's result~\cite{RN101}, our theorem can be rephrased as:
\begin{corollary}
    \(\mathsf{P}^{-} + \mathsf{I}\Sigma_0+\mathsf{B}\Sigma_1 + \mathsf{Exp} \vdash \mathsf{I}\Sigma_1 \leftrightarrow\) There exists a bi-isolated proper d.c.e.\ degree below \(\bm{0}'\).\qed{}
\end{corollary}

Wu's proof~\cite{RN172} of the existence of a bi-isolated degree relies on the existence of non-cupping degrees: a noncomputable c.e.\ degree \(\bm{a}\) is \emph{non-cupping} (or, \emph{non-cuppable} as in~\cite{RN363}) if there does not exist a c.e.\ degree \(\bm{b}<\bm{0}'\) such that \(\bm{a}\vee\bm{b}=\bm{0'}\). However, the existence of such degrees is equivalent to \(\mathsf{I}\Sigma_2\) over the base theory \(\mathsf{P}^- + \mathsf{B}\Sigma_2\)~\cite{RN363}; non-cupping degrees do not exist in every model of \(\mathsf{I}\Sigma_1\).
Therefore, the original path to proving the existence of bi-isolated d.c.e.\ degrees does not work in models of \(\mathsf{I}\Sigma_1\).
In this paper, we provide a direct proof of the existence of bi-isolated d.c.e.\ degrees in models of \(\mathsf{I}\Sigma_1\), using only finite injury priority methods.

The key observation is that the strategy for satisfying the isolation requirements is \emph{essentially} similar to the standard Friedberg-Muchnik requirements, and we combine techniques from~\cite{RN356} to prove the existence of bi-isolated d.c.e.\ degrees in models of \(\mathsf{I}\Sigma_1\).
In the classical Friedberg-Muchnik construction, each requirement searches for a specific computation to preserve and, once found, immediately acts to ensure its preservation. If such a computation does not exist, the requirement is simply satisfied automatically. Analogously—though with somewhat more complexity—a single isolation requirement (denoted as an \(N\)-requirement below) must search for multiple potential computations and aim to preserve at least one of them, often with a delay as we wait for a suitable candidate to appear. If a candidate does not appear, we proceed to construct an alternative functional \(W = \Delta^A\) as described below.

Since we are working in models of \(\mathsf{I}\Sigma_1\), we can use the finite injury argument in the usual way for the most part, except that the following lemma for \(n=1\) is needed to verify that our construction works in models of \(\mathsf{I}\Sigma_1\).

\begin{lemma}[H. Friedman]\label{lem:Friedman}
    Let \(n \geq 1\) and \(\mathcal{M} \models \mathsf{P}^- + \mathsf{I}\Sigma_n\). Then every bounded \(\Sigma_n\) set is \(\mathcal{M}\)-finite, and every partial \(\Sigma_n\) function maps a bounded set to a bounded set.
\end{lemma}

\section{Proof of Theorem~\ref{thm:main1}}\label{sec:proof}
We build a d.c.e.\ set \(D\) and a c.e.\ set \(A \le_T D\) such that the following requirements will be met:
\begin{itemize}
    \item \(N_e : W_e = \Psi_e^D \to W_e = \Delta^A\),
    \item \(R_e : D = \Phi_e^{W_e} \to K = \Gamma^{W_e}\), and
    \item \(P_e : D \neq \Theta_e^{A}\).
\end{itemize}
Here, \(\{\Psi_e\}_{e \in M}\), \(\{\Phi_e\}_{e \in M}\), and \(\{\Theta_e\}_{e \in M}\) are fixed lists of Turing functionals in our model \(M\).
We take \(A\) to be the Lachlan set \(L(D)\) of \(D\) (to be defined below).
The set \(K\) is taken to be \(\{2e+1\mid \Phi_e(e)\downarrow\}\), a subset of odd numbers.

These requirements are sufficient. As the Lachlan set of \(D\), \(A\le_T D\) (see below). The \(P_e\)-requirements imply \(A<_T D\). The \(N_e\)-requirements show that \(D\) is isolated from below by \(A\). To see \(D<_T K\), suppose otherwise that \(D\equiv_T K\).
Applying the Sacks Splitting Theorem to get \(K=D_0\oplus D_1\), we have 
\[
    K\equiv_T D_0\oplus D_1\le_T A<_T K,
\] 
a contradiction. Therefore, we have \(A<_T D<_T K\). The \(R_e\)-requirements show that \(D\) is isolated from above by \(K\).

\subsection{Preliminaries}\label{sec:preliminaries}
We will use a priority tree to organize the construction:
The priority tree is \(\cT = {\{0\}}^{<M}\), where we assign \(\alpha\) to \(N_e\) if \(\abs{\alpha}=3e\), \(R_e\) if \(\abs{\alpha}=3e+1\), and \(P_e\) if \(\abs{\alpha}=3e+2\).
At stage \(s\), we let the root of \(\cT\) act.
When \(\alpha\) acts, it decides whether \(\alpha\concat 0\) acts or we stop the current stage.
If \(\alpha\) acts at stage \(s\), then we say \(s\) is an \(\alpha\)-stage.
An \(\alpha\)\nbd{}stage~\(s\) is an \(\alpha\)-expansionary stage if \(\ell_\alpha(s)>\ell_\alpha(t)\) for each \(\alpha\)\nbd{}expansionary stage \(t<s\), where 
\[
    \ell_\alpha(s) = \begin{cases}
        \max \{ y: \forall x \le y (W_{e,s}(x)= \Psi_{e,s}^{D_s}(x)\downarrow )\}, & \text{if } \abs{\alpha}=3e,  \\

        \max \{ y: \forall x \le y (D_s(x) = \Phi_{e,s}^{W_{e,s}}(x)\downarrow )\}, & \text{if } \abs{\alpha}=3e+1.
    \end{cases}
\]

For a given enumeration \(\{D_s\}_{s\in M}\) of a d.c.e.\ set \(D\), we define
\[
    D_s^\sharp(x)= \abs{\{t\le s\mid D_{t}(x)\neq D_{t-1}(x)\}}.
\]
Note that \(D_s(x)=D_s^\sharp(x)\bmod 2\).
Let \(\sigma\) be an \(M\)-finite sequence.
The d.c.e.\ set \(D\) is \emph{unrestorable to} \(\sigma\) at stage~\(s\) if \(\exists x (D_s(x)\neq \sigma(x)\land D_s^\sharp(x)=2)\), and \emph{restorable} otherwise.
The \emph{Lachlan set} of \(D\) is a c.e.\ set
\[
   L(D)=\{\tup{x,s}\mid D_{s-1}^\sharp(x) = 0 \land D_s^\sharp(x)=1\land D(x) = 0\}\le_T D.
\]
We write \(x^*=\tup{x,s}\) with \(s\) such that \(D_{s-1}^\sharp(x)=0\) and \(D_s^\sharp(x)=1\). 
An enumeration of \(L(D)\) is to enumerate \(x^*\) into \(L(D)\) at stage~\(s\) for the least \(s\) such that \(D_s^\sharp(x)=2\).

We will also use some miscellaneous notations for future use:
Let \(\sigma\) be an \(M\)-finite sequence and let \(X\subseteq M\) be a set, viewed as a binary sequence. We write \(\sigma\subseteq X\) if \(\sigma\) is an initial segment of \(X\). 
% We let \([\sigma]=\{X\in \{0,1\}^{M}\mid \sigma\subseteq X\}\). Thus, \(\sigma\subseteq X\) if and only if \(X\in [\sigma]\). We will use whichever is most convenient in the context.

We sometimes use the standard notation \(\beta[s]\) to indicate that the expression \(\beta\) is evaluated with respect to stage~\(s\). Sometimes, we omit \([s]\) when the context is clear.

\subsection{\texorpdfstring{\(P_e\)-Strategy in Isolation}{P-strategy in isolation}}\label{sec:P-strategy-in-isolation}
Let \(\alpha\) be the \(P_e\)-node. \emph{In this section, we assume that \(\alpha\) is not going to be initialized.} This is a standard Friedberg-Muchnik strategy.

At stage~\(s\),
\begin{enumerate}[label=(P\arabic*), ref=(P\arabic*)]
    \item If \(\alpha\) is \emph{satisfied}, we let \(\alpha\concat 0\) act.
    \item If the diagonalizing witness \(w\) has not been picked, we pick a fresh one.
    \item If \(\Theta_{e,s}^{A_s}(w)\uparrow\), we let \(\alpha\concat 0\) act.
    \item \label{it:P_satisfied_1} If \(\Theta_{e,s}^{A_s}(w)\downarrow = 1\), we claim that \(\alpha\) is \emph{satisfied}, initialize all nodes below \(\alpha\), and stop the current stage. (We shall prevent \(w\) from entering \(D\).)
    \item \label{it:P_satisfied_2} If \(\Theta_{e,s}^{A_s}(w)\downarrow = 0\), we enumerate \(w\) into \(D\) and claim that \(\alpha\) is \emph{satisfied}, initialize all nodes below \(\alpha\), and stop the current stage. (We do not let others extract \(w\) from \(D\).)
\end{enumerate}
In Items~\ref{it:P_satisfied_1} and~\ref{it:P_satisfied_2}, we want to preserve the computation \(\Theta_e^A(w)\downarrow\) found at stage \(s\), where the use is \(l = \theta_e^A(w)\). 
Recall that \(A\) is the Lachlan set of \(D\). 
This computation is injured only if some \(x^* < l\) is enumerated into \(A\) (necessarily, \(D_s(x)=1\)), which happens when \(x\) is extracted from \(D\). 
However, we will see that \emph{all extractions are due to an \(N\)-node}, and whenever an extraction occurs, all nodes below it are initialized (see Remark~\ref{rmk:r1}).
Under this assumption, we have the following lemma.
\begin{lemma}\label{lem:P_satisfied}
    Let \(\alpha\) be a \(P_e\)-node. 
    Suppose \(\alpha\) is not initialized after some stage \(s_0\). Then the \(P_e\)-requirement is satisfied. 
\end{lemma}

\subsection{\texorpdfstring{\(R_e\)-Strategy in Isolation}{Re-strategy in isolation}}\label{sec:R-strategy-in-isolation}
Let \(\alpha\) be the \(R_e\)-node. \emph{In this section, we assume that \(\alpha\) is not going to be initialized.}
The node \(\alpha\) will build a local functional \(\Gamma\) such that \(\Gamma^{W_e}(x)=K(x)\) for all \(x\in M\), provided that \(\Phi_e^W = D\). 
To ensure the correctness of each \(\Gamma^{W_e}(x)=K(x)\), before defining \(\Gamma^{W_e}(x)\), we pick a fresh number \(d_{e,x}\), called the \emph{agitator} for \(x\), and wait for an \(\alpha\)-expansionary stage \(s\) when \(d_{e,x}<\ell_\alpha(s)\). 
Then we define \(\Gamma^{W_e}(x)=K_s(x)\) with use \(\gamma(x)=\varphi_e^{W_e}(d_{e,x})\). 
Whenever \(0=\Gamma^{W_e}(x)\neq K(x) =1\), we enumerate \(d_{e,x}\) into \(D\) to undefine \(\gamma_e^{W_e}(x)\), which will occur at the next \(\alpha\)-expansionary stage, if there is one.

The agitator \(d_{e,x}\) is either \emph{undefined}, \emph{defined}, \emph{active}, \emph{enumerated}, or \emph{obsolete}, as described below.

\textbf{\(R_e\)-strategy:} For all \(x\in M\), the agitator \(d_{e,x}\) is initially \emph{undefined}. At stage~\(s\),
\begin{enumerate}[label=(R\arabic*), ref=(R\arabic*)]
    \item If \(s\) is not an \(\alpha\)-expansionary stage or \(\ell_\alpha(s)<d_{e,x}\) for some \(d_{e,x}\) that is defined, we let \(\alpha\concat 0\) act (and skip the instructions below).
\end{enumerate}
\emph{Below, we assume \(s\) is an \(\alpha\)-expansionary stage and \(\ell_\alpha(s)\ge d_{e,x}\) for all \(d_{e,x}\) that are defined.}
\begin{enumerate}[label=(R\arabic*), ref=(R\arabic*), resume]
    \item\label{it:extract} For each \(x\) such that \(D_s(d_{e,x})=1\), if \(x\) is even, we extract \(d_{e,x}\); if \(x\) is odd, we keep \(d_{e,x}\) in \(D\). In both cases, we let \(d_{e,y}\) be \emph{undefined} for all \(y\ge x\).
    \item\label{it:enumerate} Let \(x\) be the least (if any) such that \(\Gamma^{W_{e,s}}(x)\downarrow\neq K_s(x)\). We enumerate \(d_{e,x}\) into \(D\) so that \(D(d_{e,x})=1\). We stop the current stage.
    \item Let \(y\) be the least (which always exists) such that \(\Gamma^{W_{e,s}}(y)\uparrow\). For each \(x\) with \(y\le x < \ell_\alpha(s)\), 
    \begin{enumerate}[label=(R\arabic{enumi}\alph*), ref=(R\arabic{enumi}\alph*)]
        \item\label{it:easy_case} If \(K_s(x)=1\), we define \(\Gamma^{W_{e,s}}(x)=1\) with use \(\gamma(x)=0\). \(d_{e,x}\) is claimed to be \emph{obsolete}.
        \item\label{it:get_defined} If \(K_s(x)=0\) but \(d_{e,x}\) is undefined, we let \(d_{e,x}\) be \emph{defined} with a fresh number.
        \item\label{it:hard_case} If \(K_s(x)=0\) and \(d_{e,x}\) is defined, we define \(\Gamma^{W_{e,s}}(x)=0\) with use \(\gamma(x)=\varphi_e^W(d_{e,x})\). We say \(d_{e,x}\) is \emph{active}.
    \end{enumerate}
    Then we let \(\alpha\concat 0\) act.
\end{enumerate}

Note that we are defining the \(\Gamma\)-functional ourselves in an orderly manner, so if \(\Gamma^{W_e}(x)\uparrow\) at a stage, then \(\Gamma^{W_e}(y)\uparrow\) for all \(y > x\) at the same stage. Thus, Items~\ref{it:easy_case} and~\ref{it:hard_case} are legitimate. We note that for \(d_{e,x}\) with odd \(x\), an \(R\)-node is responsible for enumerating it into \(D\) as in Item~\ref{it:enumerate} (recall, \(K\) is a subset of odd numbers), and never extracts it; an \(N\)-node is responsible for extracting it as in Item~\ref{it:dc_found} below.
For \(d_{e,x}\) with even \(x\), an \(N\)-node is responsible for enumerating it into \(D\) as in Item~\ref{it:N_enter_phase2} below and an \(R\)-node is responsible for extracting it as in Item~\ref{it:extract}.

If we focus on a particular \(d_{e,x}\) with \(x\) odd, a typical lifespan is as follows:
At first, it is undefined.
In Item~\ref{it:get_defined}, \(d_{e,x}\) is \emph{defined} with a fresh number.
At the next \(\alpha\)-expansionary stage, it is active in Item~\ref{it:hard_case}.
At the following \(\alpha\)-expansionary stage, \(x\) is enumerated into \(K\) and therefore \(d_{e,x}\) is enumerated into \(D\) in Item~\ref{it:enumerate} (\(\Gamma(x)\) is incorrect).
Then, at the next \(\alpha\)-expansionary stage, it is undefined in Item~\ref{it:extract} (this is when we undefine \(\Gamma(x)\)), and then in Item~\ref{it:easy_case}, it becomes \emph{obsolete}.
Thus, \(\Gamma^W(x) = K(x)\) is always correct.

Let us preview the conflicts between the \(R\)-node \(\alpha\) and the \(N\)-node \(\beta\), with \(\alpha\subseteq \beta\).
Essentially, \(\beta\) wants to protect a \emph{diagonalizing computation} \(\sigma\subseteq D\) with \(\Psi^{\sigma}(y)=0\neq W(y)=1\). This computation is injured if \(\alpha\) enumerates some \(d_{e,x}<\abs{\sigma}\) into \(D\) with \(x\) odd. The best \(\beta\) can expect is that, for some fixed even number \(k\), those \(d_{e,x}\) with \(x<k\) are allowed to injure \(\beta\) (as will be seen, this happens only \(M\)-finitely many times), but for \(x>k\), we would like to have \(d_{e,x}>\abs{\sigma}\). To achieve this, \(\beta\) enumerates \(d_{e,k}\) into \(D\). 
There are two cases:
\begin{itemize}[leftmargin=*]
\item \(\alpha\) is waiting for its expansionary stage. Meanwhile, \(\alpha\) poses no threat to other nodes. Then \(\beta\) simply searches for another diagonalizing computation, which will not be threatened by \(\alpha\).
\item \(\alpha\) has an expansionary stage. Then \(\Gamma(k)\) is undefined and, in Item~\ref{it:extract} of the \(\alpha\)-strategy, \(d_{e,k}\) is extracted and undefined. Then in Item~\ref{it:get_defined}, \(d_{e,k}\) is defined with a fresh number \(>\abs{\sigma}\).
\end{itemize}
We return to the discussion of the \(R_e\)-strategy.

\begin{lemma}\label{lem:R1}
    Let \(\alpha\) be the \(R_e\)-node. Suppose \(\alpha\) is not initialized after some stage \(s_0\). Suppose \(\Phi_e^{W_e} = D\). Then we have \(\Gamma^{W_e}=K\).
\end{lemma}
\begin{proof}
    Since \(\Phi_e^W = D\), there are \(M\)-infinitely many \(\alpha\)-expansionary stages.
    Let \(x\in M\) be given.

    \textbf{Case 1.} Item~\ref{it:easy_case} occurs.
    Then \(\Gamma^{W_e}(x) = K(x) = 1\).

    \textbf{Case 2.} Item~\ref{it:easy_case} does not occur, and for some \(\alpha\)-expansionary stage \(s_1 > s_0\), we have \(\Gamma^{W_{e,s_1}}(x)\downarrow \neq K_{s_1}(x)\).
    Then Item~\ref{it:enumerate} occurs, and so \(d_{e,x}\) is enumerated into \(D\) at \(s_1\). Let \(s_2\) be the next \(\alpha\)-expansionary stage; we have\footnote{See Remark~\ref{rmk:r2} for a future modification}
    \[
        0 = D_{s_1}(d_{e,x}) = \Phi_{e,s_1}^{W_{e,s_1}}(d_{e,x})\downarrow \neq \Phi_{e,s_2}^{W_{e,s_2}}(d_{e,x})\downarrow = D_{s_2}(d_{e,x}) = 1.
    \]
    Hence, \(\Gamma^{W_e}(x)\uparrow\) at stage \(s_2\). Then Item~\ref{it:easy_case} occurs, contradicting the hypothesis.

    \textbf{Case 3.} Item~\ref{it:easy_case} does not occur and, for all \(\alpha\)-expansionary stages \(s\), whenever \(\Gamma^{W_{e,s}}(x)\downarrow\), it equals \(K_{s}(x)\). Let \(s_1 > s_0\) be the stage after which \(d_{e,x}\) is never undefined. \emph{(The existence of such a stage involves the conflicts with \(N\)-nodes and will be proved in Lemma~\ref{lem:dex}.)} Therefore, \(d_{e,x}\) is fixed after \(s_1\). Since \(d_{e,x}\) can be enumerated or extracted at most once, there is another stage \(s_2 \geq s_1\) after which \(D(d_{e,x}) = D_{s_2}(d_{e,x})\).
    We note that \(\Phi_e^{W_e}(d_{e,x})\downarrow\) implies there exists a stage \(s_3 > s_2\) after which \(\Phi_e^{W_e}(d_{e,x})\) is never undefined again. Let \(s_4 > s_3\) be the next \(\alpha\)-expansionary stage. \(\Gamma^{W_e}(x)\) is defined with use \(\gamma(x) = \varphi_e^{W_e}(d_{e,x})\) at the end of \(s_4\) (as in Item~\ref{it:hard_case}). After \(s_4\), \(\Gamma^{W_e}(x)\) is never undefined. Thus, \(\Gamma^{W_e}(x) = K(x)\).

    Hence, in all cases, \(\Gamma^{W_e}(x) = K(x)\). As \(x\) is arbitrary, \(\Gamma^{W_e} = K\).
\end{proof}
 
Together with Lemma~\ref{lem:finite_injury} and Lemma~\ref{lem:R1}, we obtain the following lemma:
\begin{lemma}\label{lem:R_satisfied}
    For each \(e\in M\), \(R_e\) is satisfied. \qed{}
\end{lemma}

\subsection{\(N_e\)-Requirement in Isolation}
Recall that \(A = L(D)\).

The idea for the \(N_e\)-node \(\alpha\) is to define and maintain \(W_e = \Delta^A\) during its expansionary stages. Under the assumption that there are infinitely many \(\alpha\)-expansionary stages, there are three possible outcomes:
\begin{enumerate}[label=(O\arabic*), ref=(O\arabic*)]
    \item \label{it:outcome_1} \(\Delta^A = W_e\) is total and correct.
    \item \label{it:outcome_2} If, for some \(x\), \(\Delta^A(x) \neq W_e(x)\), then \(D\) is restored to some diagonalizing computation \(\sigma\) with \(\Psi_e^\sigma(x) \neq W_e(x)\).
    \item \label{it:outcome_3} If, for some \(x\), \(\Delta^A(x)\) diverges, then \(\Psi_e^D(x)\) diverges.
\end{enumerate}
In all cases, the \(N_e\)-requirement is satisfied. The main focus for \(\alpha\) is to keep \(\Delta^A(x)\) correct.
To elaborate, suppose we are to define \(\Delta^A(x)\) at some \(\alpha\)-expansionary stage~\(s\). We let \(\sigma = D_s \res \psi_{e,s}(x)\) and \(\tau = A_s \res l\) for some fresh number \(l\). 
In particular, for those \(z\) with \(\sigma(z)=1\), we have \(z^*<k\).
We have \(\Psi_e^{\sigma}(x) = W_{e,s}(x)\), and we define \(\Delta^{\tau}(x) = W_{e,s}(x)\) and \((x; \sigma, \tau)\) as the \emph{diagonalizing pair}.
At a later stage~\(t>s\), we will have three cases:
\begin{enumerate}
    \item \(D_t\) becomes unrestorable to \(\sigma\). Let \(z\) be the number such that \(D_t(z) \neq \sigma(z)\), with \(D_t^\sharp(z) = 2\). 
    Then \(z^*\) is enumerated into \(A_t\) and therefore \(A_t \notin [\tau]\). 
    Therefore, \(\Delta^{A_t}(x) \uparrow\).
    In this case, we cancel the diagonalizing pair and start anew. If this happens infinitely many times, then we have Item~\ref{it:outcome_3}.
    \item If \(D_t\) is still restorable to \(\sigma\) and \(\Delta^{A_t}(x)\uparrow\) (due to some noise), we define \(\Delta^{A_t}(x)=W_{e,t}(x)\) with the same use \(\abs{\tau}\) and update the diagonalizing pair to \((x;\sigma, A_t\res \abs{\tau})\) (the previous one is canceled). In other words, if \(\Delta^A\) is not incorrect yet, we have to ensure the totality of \(\Delta^A\), corresponding to Item~\ref{it:outcome_1}.
    \item If \(D_t\) is still restorable to \(\sigma\) and \(\Delta^{A_t}(x)\downarrow \neq W_{e,t}(x)\), then we claim that \(\sigma\) is a diagonalizing computation, and we restore \(D_t\) to \(\sigma\) so that \(W_{e,t}(x) \neq \Psi_e^{D_t}(x)\). This corresponds to Item~\ref{it:outcome_2}.
\end{enumerate}

The above strategy is referred to as one cycle; the purpose of this cycle is to maintain its version of \(\Delta^A\) or to produce a diagonalizing computation. We may need multiple cycles, and the \(k\)-th cycle is referred to as a \(\cyc(k)\)-module.

\textbf{\(\cyc(k)\)-module} at stage \(s\):
\begin{enumerate}[label=(C\arabic*), ref=(C\arabic*)]
    \item If \(s\) is not an \(\alpha\)-expansionary stage, we let \(\alpha\concat 0\) act and skip the instructions below.
\end{enumerate}
\emph{Below, we assume \(s\) is an \(\alpha\)-expansionary stage.}
\begin{enumerate}[label=(C\arabic*), ref=(C\arabic*), resume]
    \item Let \(x\le s\) be the least (which always exists) such that \(\Delta^{A_s}(x)\uparrow\) or \(\Delta^{A_s}(x)\downarrow \neq W_{e,s}(x)\).
    \begin{enumerate}[label=(C\arabic{enumi}\alph*), ref=(C\arabic{enumi}\alph*)]
        \item If \(\Delta^{A_s}(x)\uparrow\), for each \(y\) with \(x\le y<s\),
            \begin{enumerate}[label=(C\arabic{enumi}\alph{enumii}.\roman*), ref=(C\arabic{enumi}\alph{enumii}.\roman*)]
                \item\label{it:cyc_noise} if there is an \((y;\sigma,\tau)\) for some \(\sigma\) and \(\tau\) has been defined such that \(D_s\) is restorable to \(\sigma\), then we define \(\Delta^{A_s}(y)=W_{e,s}(y)\) with the same use as \(l=|\tau|\) and define \((y;\sigma,A_s\upharpoonright l)\) as the diagonalizing pair (the old one is canceled).
                \item\label{it:cyc_div} otherwise, we define \(\Delta^{A_s}(y)=W_{e,s}(y)\) with a fresh use \(z\) and define \((y;D_s\res \psi_{e,s}(y),A_s\upharpoonright z)\) as the diagonalizing pair.
            \end{enumerate}
            Then we let \(\alpha\concat 0\) act.
        \item\label{it:dc_found} If \(\Delta^{A_s}(x)\downarrow \neq W_{e,s}(x)\), let \((x;\sigma,\tau)\) with \(\tau\subseteq A_s\) be the diagonalizing pair. We restore \(D_s\) to \(\sigma\) and claim that \(\sigma\) is the \emph{diagonalizing computation} for \(x\), denoted as \(\sigma=\dc(k)\). 
    \end{enumerate}
\end{enumerate}
Note that Item~\ref{it:cyc_div} happens if the diagonalizing pair is not defined or \(D\) is unrestorable to \(\sigma\). Since we define the \(\Delta\)-functional in an orderly manner, \(\Delta^{A_s}(x)\uparrow\) implies \(\Delta^{A_s}(y)\uparrow\) for \(y>x\).
Once a diagonalizing computation \(\sigma=\dc(k)\) is found, it may still have conflicts with \(R\)-nodes above \(\alpha\). Recall from the discussion above Lemma~\ref{lem:R1}.

\subsection{\texorpdfstring{Conflicts between \(N_e\)-Strategy and \(R_i\)-Strategy}{Conflicts between N-strategy and R-strategy}}\label{sec:conflicts-between-N-and-R}
Let \(\alpha\) be an \(N_e\)-node. There are \(e\) many \(R\)-nodes above \(\alpha\), namely the \(R_i\)-nodes for \(i<e\). Let \(\sigma=\dc(k)\) be the diagonalizing computation for \(x\) found in the \(\cyc(k)\)-module. We restore \(D\) to \(\sigma\) immediately. Then, for each \(i<e\),
\begin{itemize}
    \item if \(d_{i,x}\) with \(x<2(e-i)\) is enumerated or extracted, then \(\alpha\) is initialized. It will be shown that this happens to \(\alpha\) only \(M\)-finitely many times.
    \item for \(x>2(e-i)\), we would like to have \(d_{i,x}>\abs{\sigma}\) so future enumeration of this number will not injure the diagonalizing computation \(\sigma\). 
\end{itemize}
To achieve the second point, \(\alpha\) itself enumerates \(d_{i,2(e-i)}\) for each \(i<e\) with \(d_{i,2(e-i)}\) defined, into \(D\), in order to undefine \(\Gamma_i^{W_e}(d_{i,2(e-i)})\). The diagonalizing computation \(\sigma\) is injured by the enumeration of \(d_{i,2(e-i)}\). Once the \(R_i\)-node has an expansionary stage, \(\Gamma_i^{W_e}(d_{i,2(e-i)})\) becomes undefined, \(d_{i,2(e-i)}\) is extracted, and for each \(x \geq 2(e-i)\), the \(R_i\)-node will pick a fresh number \(d_{i,x} > |\sigma|\). Once this happens for all \(i < e\), \(D\) is restored to \(\sigma\) and it is cleared of any future conflicts with \(R_i\)-nodes for \(i < e\).
Before this is done, \(\alpha\) starts the \(\cyc(k+1)\)-module.

To summarize, \(\cyc(k)\) can be in \emph{Phase 1}, \emph{Phase 2}, \emph{accomplished}, or \emph{initialized}. In \emph{Phase 1}, \(\cyc(k)\) searches for a diagonalizing computation. Once it is found, \(\cyc(k)\) enters \emph{Phase 2}. Then, we enumerate all the \(d_{i,2(e-i)}\) (if defined) for \(i<e\) into \(D\). 

Let \(d_i^*=d_{i,2(e-i)}\) record the current version of \(d_{i,2(e-i)}\) for \(i<e\). We define
\[
\Upsilon_k(s)=\{i<e\mid D_s(d_i^*)=1\}.
\]
If for some stage \(s\), \(\Upsilon_k(s)=\varnothing\), then \(D\) is restored to \(\sigma\) and \(\cyc(k)\) is \emph{accomplished}.
If \(\cyc(k)\) is initialized, then its version of \(\Delta^A\) is discarded and \(\Upsilon_k\) is canceled.

\textbf{\(N_e\)-strategy:} Let \(\alpha\) be the \(N_e\)-node. Initially, \(\cycCount(e)=0\) and \(\cyc(k)\) is initialized for all \(k\in M\). At stage \(s\) (with \(s^*<s\) the last \(\alpha\)-stage),
\begin{enumerate}[label=(N\arabic*), ref=(N\arabic*)]
    \item\label{it:N_init} If, for some \(d_{i,x}\) with \(i<e\) and \(x<2(e-i)\), we have \(D_{s^*}(d_{i,x})\neq D_s(d_{i,x})\), then initialize \(\alpha\) and all nodes below \(\alpha\). Stop the current stage.
    \item\label{it:N_init_cyc} If, for some (least) \(k\), \(\Upsilon_k(s)\neq \Upsilon_{k}(s^*)\), initialize all \(\cyc(j)\)-modules for \(j>k\) and all nodes below \(\alpha\).
    \item\label{it:N_accomlished} If, for some (least) \(k\), \(\Upsilon_k(s)=\varnothing\), then declare \(\cyc(k)\) to be \emph{accomplished}, and let \(\alpha\concat 0\) act.
    \item\label{it:N_enter_phase2} Otherwise, let \(k\) be the least such that \(\cyc(k)\) is \emph{initialized} or in \emph{Phase 1}. Let \(\cyc(k)\) (continued) act. If Item~\ref{it:dc_found} occurs in the \(\cyc(k)\)-module, proceed as follows:
    \begin{itemize}
        \item For each \(i<e\),
        \begin{itemize}
            \item If \(d_{i,2(e-i)}\) is active, enumerate it into \(D\).
            \item If \(d_{i,2(e-i)}\) is defined, make it undefined.
            \item If \(d_{i,2(e-i)}\) is already enumerated, do nothing.
        \end{itemize}
        \item Define \(\Upsilon_k(s)\) as above. \(\cyc(k)\) now enters \emph{Phase 2}.
        \item Add \(1\) to \(\cycCount(e)\).
        \item Initialize all nodes below \(\alpha\) and stop the current stage.
    \end{itemize} 
\end{enumerate}

\begin{remark}\label{rmk:r1} 
    The extraction of \(d_{i,2(e-i)}\) that happens in Item~\ref{it:extract} for the \(R_i\)-strategy will cause Item~\ref{it:N_init_cyc} to occur, and all nodes below the \(N_e\)-node are therefore initialized. This phenomenon confirms the assumption discussed above Lemma~\ref{lem:P_satisfied}.
\end{remark}
\begin{remark}\label{rmk:r2}
    In the proof of Lemma~\ref{lem:R1}, where \(\alpha\) is the \(R_e\)-node, we deliberately omit a subtle case to make the main idea clear.
    In fact, after \(s_1\), when \(d_{e,x}\) is enumerated,
    an \(N_j\)-node with \(j>e\) can extract it at some stage \(s_{1.5}>s_1\) when Item~\ref{it:dc_found} occurs.
    According to Item~\ref{it:N_enter_phase2}, we immediately enumerate \(d_{e,2(j-e)}\) into \(D\).
    We have \(d_{e,2(j-e)}<x\); otherwise, the \(N_j\)-node would be initialized at stage \(s_1\) when the \(R_e\)-node enumerates \(d_{e,x}\) into \(D\).
    In such a case, we have \(\Phi_e^W(d_{e,2(j-e)})[s_1]\downarrow \neq \Phi_e^W(d_{e,2(j-e)})[s_2]\downarrow\), which implies \(\Gamma^W(2(j-e))\uparrow\) and hence \(\Gamma^W(x)\uparrow\) at stage \(s_2>s_{1.5}\), leading to the same contradiction as in the proof.
\end{remark}

\begin{lemma}\label{lem:N_count}
    For each \(e\in M\), \(\cycCount(e)\le 2^e\).
\end{lemma}
\begin{proof}
    We prove this by induction on \(e\). For \(e=0\), the \(N_0\)-node has no \(R\)-nodes above it. Therefore, once \(\cyc(0)\) enters \emph{Phase 2}, it becomes \emph{accomplished} and no more cycles are needed. Therefore, \(\cycCount(0)\le 1=2^0\).
    
    Let \(\alpha\) be the \(N_e\)-node. If \(\cyc(0)\) never enters \emph{Phase 2}, there will be no other cycles. So we assume \(\cyc(0)\) enters \emph{Phase 2} at some stage \(s_0\).
    Note that if \(i\in \Upsilon_0(s)\), then \(d_{i,2(e-i)}\) is still enumerated at stage \(s\) and the \(R_i\)-node is still waiting for an expansionary stage and therefore poses no more threat to \(\alpha\). In other words, initially, \(\alpha\) has \(e\) many \(R\)-nodes that have potential conflicts with \(\alpha\). If \(|\Upsilon_0(s)|=j\) for some stage \(s\), then there are only \(e-j\) many \(R\)-nodes that have potential conflicts with \(\alpha\). \(\alpha\) now has a situation analogous to that of an \(N_{e-j}\)-node with \(\cycCount(e-j)\le 2^{e-j}\). Therefore, there will be at most \(2^{e-j}\) more cycles that can ever enter \emph{Phase 2} after \(s\), while \(|\Upsilon_0(s)|=j\) remains unchanged.

    Since \(|\Upsilon_0(s)|\) is non-increasing, ranging from \(e\) down to \(0\), and whenever \(|\Upsilon_0(s)|=0\), \(\cyc(0)\) becomes \emph{accomplished} and no more cycles are needed. Therefore, 
    \begin{align}
        \cycCount(e)&\le 1+\cycCount(0) + \cycCount(1) + \cdots + \cycCount(e-1)\\
        &\le 1+2^0+2^1+\cdots+2^{e-1}=2^e,
    \end{align}
    where the first 1 is for \(\cyc(0)\).

    This completes the proof.
\end{proof}

\begin{lemma}\label{lem:N_satisfied}
    Let \(\alpha\) be the \(N_e\)-node. Suppose that \(\alpha\) is not initialized after \(s_0\). Then, the \(N_e\)-requirement is satisfied.
\end{lemma}
\begin{proof}
    Suppose that \(W_e=\Psi_e^D\). Then \(\alpha\) has \(M\)-infinitely many expansionary stages.
    Since \(\cycCount(e)\le 2^e\) (Lemma~\ref{lem:N_count}), there exists a stage \(s_1>s_0\) after which the status of all cycles of \(\alpha\) remains unchanged. 
    If, for some \(k\), we have \(\Upsilon_k(s_1)=\varnothing\), let \(\sigma=\dc(k)\) be the diagonalizing computation for \(x\). Then \(\sigma\subseteq D\) and
    \[
        \Phi_e^\sigma(x)\downarrow \neq W_e(x),
    \]
    contradicting the hypothesis that \(W_e=\Psi_e^D\). Therefore, no cycles are accomplished. Let \(\cyc(k)\) be the (unique) cycle which is in \emph{Phase 1} at stage \(s_1\). Item~\ref{it:dc_found} does not occur in the \(\cyc(k)\)-module. Fix an arbitrary \(x\in M\).
    Recall that Item~\ref{it:cyc_div} occurs if the diagonalizing pair is not defined (which happens only once), or if \(D\) is unrestorable to \(\sigma\). Therefore, if Item~\ref{it:cyc_div} happens \(M\)-infinitely often, then \(D\) is unrestorable to \(\sigma\) \(M\)-infinitely often. This necessarily implies \(\Psi_e^D(x)\uparrow\), contradicting the hypothesis that \(W_e=\Psi_e^D\). Therefore, there exists a stage \(s_2>s_1\) after which Item~\ref{it:cyc_div} no longer happens, so \(\Delta^{A}(x)\downarrow\). 

    Hence, \(\Delta^A=W_e\), and this completes the proof.
\end{proof}

\subsection{Construction and Verification}
The construction proceeds stage by stage as follows:

At stage~\(s\), let the root of the priority tree \(\cT\) act. Suppose \(\alpha\) is to act.
\begin{itemize}
    \item If \(\abs{\alpha}=s\), stop the current stage.
    \item If \(\alpha\) is an \(N_e\)-node, let \(\alpha\) act according to the \(N_e\)-strategy.
    \item If \(\alpha\) is an \(R_e\)-node, let \(\alpha\) act according to the \(R_e\)-strategy.
    \item If \(\alpha\) is a \(P_e\)-node, let \(\alpha\) act according to the \(P_e\)-strategy.
\end{itemize}
It remains to prove the following:
\begin{lemma}\label{lem:finite_injury}
    For each \(\alpha\), there exists some stage \(s_0\) after which \(\alpha\) is not initialized.
\end{lemma}

In models of full arithmetic, the lemma is straightforward because all relevant inductions are available and the number of initializations can be shown to be finite. However, when working within \(I\Sigma_1\), we must provide a concrete, uniform bound on the number of times each node \(\alpha\) can be initialized in order to carry out the argument within the weaker induction available. Therefore, to establish the lemma in \(I\Sigma_1\), it suffices to effectively bound the number of initializations for each \(\alpha\).

\begin{definition}
    Let \(f(e)\), \(g(e)\), and \(h(e)\) denote \(1+\) the number of times initialization occurs to \(N_e\), \(R_e\), and \(P_e\), respectively.
\end{definition}

Estimated in a coarse manner, we have the following recurrence relations:
\begin{lemma}\label{lem:recurrence} For all \(e\in M\),
    \begin{align}
    f(0) &= 1 \tag{1} \\
    g(e) &\leq e(\cycCount(e)+1)f(e) \tag{2} \\
    h(e) &=g(e) \tag{3}\\
    f(e+1) &\leq 2h(e)+e+1 \tag{4}
    \end{align}
\end{lemma}
\begin{proof}
    \begin{enumerate}
        \item The \(N_0\)-node is never initialized, so \(f(0)=1\).
        \item While the \(N_e\)-node is not initialized, there are at most \(\cycCount(e)\) cycles, and for each cycle, Item~\ref{it:N_init_cyc} happens at most \(e\) times. Therefore, \(g(e)\le e(\cycCount(e)+1)f(e)\).
        \item The \(R_e\)-node does not initialize nodes below it. In particular, the \(P_e\)-node is not initialized by the \(R_e\)-node, so \(h(e)=g(e)\).
        \item While the \(P_e\)-node is not initialized, Item~\ref{it:P_satisfied_1} or Item~\ref{it:P_satisfied_2} happens at most once.
        We also have to count the number of times Item~\ref{it:N_init} occurs for the \(N_{e+1}\)-node. 
        Most cases are absorbed in the initializations of \(N_e\), \(R_e\), and \(P_e\).
        The other cases are caused by Item~\ref{it:enumerate}. There are at most \(e+1\) odd numbers \(<2(e+1)\). Thus, Item~\ref{it:enumerate} happens at most \(e+1\) times \emph{globally}, as \(K\) is global. Therefore \(f(e+1)\le 2h(e)+e+1\).
    \end{enumerate}
\end{proof}

\begin{lemma}
    For all \(e\in M\), \(f(e)\le 2^{e^2}\) and \(h(e)=g(e)\le e(2^e+1)2^{e^2}\).
\end{lemma}
\begin{proof}
    We only need to prove the first part, and this is proved by induction on \(e\). \(f(0)\le 1=2^{0^2}\).

    From (2) and (4) in Lemma~\ref{lem:recurrence}, we have
    \[
        f(e+1) \le 2e(\cycCount(e)+1)f(e)+e+1.
    \] 
    From the induction hypothesis \(f(e)\le 2^e\) and the basic inequality \(e<e+1\le 2^e\le 2^{e^2}\), we have 
    \begin{align*}
    f(e+1) &\le 2e(2^e+1)2^{e^2}+e+1\\
     &\le (e 2^{e+1} + 2e) 2^{e^2}+2^{e^2}\\
     &= (e 2^{e+1}+e+e+1) 2^{e^2}\\
     &\le (e 2^{e+1}+2^e+2^e) 2^{e^2}\\
     &\le (e+1)2^{e+1} 2^{e^2}\\
     &\le 2^{2e+1}2^{e^2}=2^{(e+1)^2}.
    \end{align*}
\end{proof}

\begin{proof}[Proof of Lemma~\ref{lem:finite_injury}]
    Let \(\alpha\in\cT\). Let 
    \[
    A_\alpha=\{i\mid \alpha \text{ is initialized at least \(i\) times}\},
    \]
    By Lemma~\ref{lem:recurrence}, \(A_\alpha\) is bounded. Since it is also \(\Sigma_1\), \(A_\alpha\) is \(M\)-finite (Lemma~\ref{lem:Friedman}). Let \(f:A_\alpha\to M\) map \(i\) to the stage \(s\) when \(\alpha\) is initialized for the \(i\)-th time; then \(f\) is \(\Sigma_1\)-definable, and \(\mathsf{B}\Sigma_1\) implies that the range of \(f\) is bounded by some stage \(s_0\). Then, for all stages \(s>s_0\), \(\alpha\) is not initialized.
\end{proof}

\begin{lemma}\label{lem:dex}
    For each \(e,x\in M\), there exists a stage \(s\) after which \(d_{e,x}\) is not undefined.
\end{lemma}

\begin{proof}
    We remark that \(d_{e,x}\) can be undefined only in Item~\ref{it:extract}. Therefore, we only need to consider which node is responsible for enumerating some \(d_{e,y}\) with \(y \leq x\) into \(D\).

    Suppose \(x\) is even, so \(x=2i\) for some \(i\). Thus, the \(N_{e+i}\)-node is responsible for enumerating \(d_{e,x}\) into \(D\).
    Let \(s_0\) be the stage after which both \(R_e\)- and \(N_{e+i}\)-nodes are not initialized, by Lemma~\ref{lem:finite_injury}. By the choice of \(s_0\) and Item~\ref{it:N_init}, no \(d_{e,y}\) with \(y < x\) is enumerated into or extracted from \(D\) after \(s_0\). After \(s_0\), only the \(N_{e+i}\)-node can possibly enumerate \(d_{e,x}\) into \(D\) for the last time. Therefore, there exists a stage \(s_1 > s_0\) after which \(d_{e,x}\) is not undefined.

    Suppose \(x\) is odd, so \(x=2i+1\) for some \(i\). Let \(s_0\) be the stage after which \(R_e\) is not initialized and \(d_{e,2i}\) is always defined. Since \(x\) is odd, \(d_{e,x}\) is enumerated into \(D\) only when \(x\) enters \(K\), which happens at most once. Therefore, there exists a stage \(s_1 > s_0\) after which \(d_{e,x}\) is not undefined.

    This completes the proof.
\end{proof}

We remark that our bi-isolated set \(D\) can be made low, since the lowness requirements are not more difficult to satisfy than the \(N_e\)-requirements, as can be seen in~\cite{RN356}.

\section{Open Problems}
In the literature~\cite{RN220}, there is an alternative definition of isolated degrees: a d.c.e.\ degree \(\bm{d}\) is \emph{isolated} if there exists a c.e.\ degree \(\bm{a}<\bm{d}\) and there is no other c.e.\ degree strictly between \(\bm{a}\) and \(\bm{d}\). Note that the requirement that \(\bm{d}\) be a \emph{proper} d.c.e.\ degree is not included in this definition or in Definition~\ref{def:isolated_below}. Therefore, the two definitions are equivalent because of the Sacks Density Theorem~\cite{RN153}.
While Groszek, Mytilinaios, and Slaman~\cite{RN320} showed that \(\mathsf{B}\Sigma_2\) suffices to prove the Sacks Density Theorem, the exact strength remains unclear:

\begin{question}
Does the Sacks Density Theorem fail in some model of \(\mathsf{I}\Sigma_1\)?
\end{question}
Depending on the answer to this question, the two versions of isolated d.c.e.\ degrees may differ from each other in some model of \(\mathsf{I}\Sigma_1\).

\bibliographystyle{plain}
\bibliography{ref.bib}

\end{document}